\documentclass[10pt,reqno]{amsart}

\usepackage{amsmath}
\usepackage{amsfonts}
\usepackage{amssymb}
\usepackage{amsthm}
\usepackage{graphicx}
\usepackage{epstopdf}
\usepackage{enumitem}
\usepackage{geometry}
\usepackage{tikz}

\definecolor{myurlcolor}{rgb}{0,0,0.7}
\usepackage{hyperref}
\hypersetup{colorlinks,
linkcolor=myurlcolor,
citecolor=myurlcolor,
urlcolor=myurlcolor}

\newtheorem{teo}{Theorem}
\newtheorem{defin}[teo]{Definition}

\newtheorem{rem}[teo]{Remark}
\newtheorem{cor}[teo]{Corollary}
\newtheorem{lem}[teo]{Lemma}

\newcommand{\beq}{\begin{equation}}
\newcommand{\eeq}{\end{equation}}

\begin{document}

\title{On the axiomatization of convex subsets of Banach spaces}

\author{Valerio Capraro}
\address{University of Neuchatel, Switzerland}
\thanks{Supported by Swiss SNF Sinergia project CRSI22-130435}
\email{valerio.capraro@unine.ch}

\author{Tobias Fritz}
\address{Institut de Ci\`{e}ncies Fot\`{o}niques, Mediterranean Technology Park, 08860 Castelldefels (Barcelona), Spain}
\thanks{Supported by the EU STREP QCS}
\email{tobias.fritz@icfo.es}



\date{\today}

\maketitle

\begin{abstract}
We prove that any convex-like structure in the sense of Nate Brown is affinely
and isometrically isomorphic to a closed convex subset of a Banach space. This
answers an open question of Brown. As an intermediate step, we identify Brown's
algebraic axioms as equivalent to certain well-known axioms of abstract
convexity. We conclude with a new characterization of convex subsets of Banach spaces.
\end{abstract}

\bigskip
\textbf{Post-publication note:} \emph{As we have learnt in the meantime\footnote{We thank Klaus Keimel for pointing out Neumann's work to us.}, another important reference on the theory of convex spaces is~\cite{Ne}, where all the results of relevance to us can already be found. In particular, this concerns Definition~\ref{convex} (in a slightly different formulation), Theorem~\ref{stone} (in the language of universal algebra~\cite[Theorem~2]{Ne}) and Lemma~\ref{lem5}, none of which is therefore original to this paper. Our main results (Corollary~\ref{embedding} and Theorem~\ref{banach}) seem to remain original.} 
\bigskip

\section{Introduction}

While studying an invariant of $\mathrm{II_1}$-factors related to Connes'
embedding conjecture, Brown~\cite{Br} found that there is a natural way of
defining convex combinations on this invariant.  However, there seemed to be no
evident embedding of this set into some linear space such that the convex
combinations are precisely those inherited from the vector space structure.
Searching for an axiomatization of those metric spaces where it makes sense to
talk about convex combinations without having any linear structure, he proposed
the notion of a convex-like structure. The obvious examples of convex-like
structures are closed convex subsets of Banach spaces. The very basic question
is whether any convex-like structure is of this form. Besides being interesting
in itself, this question has also a technical reason: there are many properties
of convex combinations which are trivial to verify in vector spaces, but are
hard to prove in the context of convex-like structures.  Here we give a
positive answer to this problem.

Actually, we obtain this result as a consequence of a more general
one: four of the five Brown's axiom, exactly the algebraic ones, are equivalent
to certain well-known axioms of abstract convexity. These were introduced by
Stone~\cite{St} and have since been discussed and sometimes rediscovered,
modulo minor variations, several
times~\cite{Fr}\cite{Gu}\cite{Mo}\cite{PR}\cite{Se} using various terminology;
here, we shall follow the notation and terminology of~\cite{Fr}.

\section{Convex-like structures and convex spaces}

In order to be precise, and also for the convenience of the
reader, we recall the definitions and the Stone embedding theorem
which we are going to use. The following two definitions are both
abstractions of the properties of convex combinations in vector
spaces.

\begin{defin}[{\cite{Br}}]\label{convex-like}
Let $(X,d)$ be a complete metric space. Take $X^{(n)} = X\times
\cdots \times X$ to be the $n$-fold Cartesian product and $\mathrm{Prob}_n$
the set of probability measures on the $n$-element set
$\{1,2,\ldots,n\}$ endowed with the $\ell_1$-metric $\| \mu -
\tilde{\mu} \| = \sum_{i=1}^n |\mu(i) - \tilde{\mu}(i) |$. We say
that $(X,d)$ has a \emph{convex-like structure} if for every $n \in
\mathbb N$ and $\mu \in \mathrm{Prob}_n$ there is given a continuous map
$\gamma_\mu \colon X^{(n)} \to X$ such that
\begin{enumerate}[label={\normalfont{($\gamma.\arabic*$)}}]
\item \label{abelian} for each permutation $\sigma \in S_n$ and
$x_1,\ldots, x_n \in X$, $$\gamma_\mu(x_1,\ldots,x_n) =
\gamma_{\mu\circ\sigma} (x_{\sigma(1)},\ldots, x_{\sigma(n)});$$

\item \label{linearity} if $x_1 = x_2$, then $\gamma_\mu(x_1,x_2,
\ldots,x_n) = \gamma_{\tilde{\mu}} (x_1, x_3, \ldots, x_n)$, where
$\tilde{\mu} \in \mathrm{Prob}_{n-1}$ is given by $\tilde{\mu} (1) = \mu(1) +
\mu(2)$ and $\tilde{\mu} (j) = \mu(j+1)$ for $2 \leq j \leq n-1$;

\item\label{dirac} if $\mu(i) = 1$, then $\gamma_\mu(x_1,\ldots,x_n) = x_i$;

\item\label{metric} The metric compatibility axiom\footnote{Brown's original metric compatibility
axiom actually consisted of two conditions. See remark~\ref{finalrem}.} for all $y_1,\ldots,y_n \in X$,
$$d (\gamma_{\mu} (x_1,\ldots,x_n), \gamma_{\mu} (y_1,\ldots,y_n))
\leq \sum_{i = 1}^n \mu(i) d(x_i, y_i);$$

\item\label{algebraic} for all $\nu \in \mathrm{Prob}_2$, $\mu \in \mathrm{Prob}_n$,
$\tilde{\mu} \in \mathrm{Prob}_m$ and $x_1,\ldots, x_n, \tilde{x}_1,\ldots,
\tilde{x}_m \in X$, $$\gamma_{\nu}(\gamma_\mu(x_1,\ldots,x_n),
\gamma_{\tilde{\mu}}( \tilde{x}_1,\ldots, \tilde{x}_m) ) =
\gamma_{\eta} ( x_1,\ldots,x_n, \tilde{x}_1,\ldots, \tilde{x}_m),$$
where $\eta \in \mathrm{Prob}_{n+m}$ is given by $\eta(i) = \nu(1)\mu(i)$, if
$1\leq i \leq n$, and $\eta(j + n) = \nu(2) \tilde{\mu}(j)$, if
$1\leq j \leq m$.
\end{enumerate}
\end{defin}

The idea behind this definition is that the $n$-ary operation $\gamma_\mu$ is supposed to stand for a convex combination with weights given by the coefficients of $\mu$:
\begin{equation}
\label{gammadef}
\gamma_\mu(x_1,\ldots,x_n) \:\widehat{=}\: \sum_{i=1}^n\mu(i)x_i\:.
\end{equation}
With this intuition, it is clear why one wants the
properties~\ref{abelian} through~\ref{algebraic} to hold.

\begin{defin}[{\cite{Fr}}]\label{convex}
A convex space is given by a set $X$ and a family of binary
operations $\{cc_\lambda\}_{\lambda\in[0,1]}$ on $X$ such that
\begin{enumerate}[label={\normalfont{(cs.\arabic*)}}]
\item\label{unitlaw} $cc_0(x,y)=x \quad \forall x,y\in X$
\item\label{idempotency} $cc_\lambda(x,x)=x \quad \forall x\in X,\:\lambda\in[0,1]$,
\item\label{commutativity} $cc_\lambda(x,y)=cc_{1-\lambda}(y,x) \quad \forall x,y\in X,\:\lambda\in[0,1]$,
\item\label{associativity} $cc_\lambda(cc_\mu(x,y),z)=cc_{\lambda\mu}(x,cc_\nu(y,z)) \quad \forall x,y,z\in X,\:\lambda,\mu\in[0,1]$,
where $\nu$ is arbitrary if $\lambda=\mu=1$ and $\nu=\frac{\lambda(1-\mu)}{1-\lambda\mu}$ otherwise.
\end{enumerate}
\end{defin}

Now an interpretation analogous to~(\ref{gammadef}) holds: the
$cc_\lambda$ simply model binary convex combinations with weight
$\lambda$:
\begin{equation}
\label{ccdef}
cc_\lambda(x,y) \:\widehat{=}\: \lambda x+(1-\lambda)y\:.
\end{equation}
Again, properties~\ref{unitlaw} through~\ref{associativity} clearly
hold for convex combinations in vector spaces.

Our first result follows now. It states that convex-like structures differ from
convex spaces just by the metric compatibility axiom~\ref{metric}. The following
equation~(\ref{gammacc}) is motivated by the correspondences~(\ref{gammadef})
and~(\ref{ccdef}).

A piece of notation: when $\mu(1)=\lambda\in[0,1]$ and $\mu(2)=1-\lambda$ are the parameters of a distribution $\mu\in \mathrm{Prob}_2$, then we also write $\gamma_{\lambda,1-\lambda}$ instead of $\gamma_\mu$.

\begin{teo}\label{equivalence}
The algebraic axioms~\ref{abelian},~\ref{linearity},~\ref{dirac} and~\ref{algebraic} of Definition
\ref{convex-like} are equivalent to the axioms of convex space in
Definition \ref{convex}. More precisely: for a given set $X$, a convex-like structure on $X$ and the structure of a convex space on $X$ mutually determine each other by the identity
\begin{equation}
\label{gammacc}
cc_{\lambda}(x,y)=\gamma_{\lambda,1-\lambda}(x,y)\:.
\end{equation}

\begin{proof}
Let us start proving that Brown's axioms~\ref{abelian},~\ref{linearity},~\ref{dirac} and~\ref{algebraic} for a convex-like structure imply
the axioms of convex spaces when the $cc_{\lambda}$ are defined as in~(\ref{gammacc}).
\begin{enumerate}
\item[\ref{unitlaw}] We have $cc_0(x,y)=\gamma_{0,1}(x,y)=y$ thanks to Brown's axiom~\ref{dirac}.
\item[\ref{idempotency}] We have $cc_\lambda(x,x)=\gamma_{\lambda,1-\lambda}(x,x)$ thanks to Brown's axiom~\ref{linearity}.
\item[\ref{commutativity}] We have
$$
cc_\lambda(x,y)=\gamma_{\lambda,1-\lambda}(x,y)=\gamma_{1-\lambda,\lambda}(y,x)=cc_{1-\lambda}(y,x)
$$
thanks to Brown's axiom~\ref{abelian}.
\item[\ref{associativity}] This is implied by the previous axioms when $\lambda=\mu=1$, so it is enough to treat the case $\lambda\mu\neq1$. We will evaluate $cc_\lambda(cc_\mu(x,y),z)$ and $cc_{\lambda\mu}(x,cc_{\frac{\lambda(1-\mu)}{1-\lambda\mu}}(y,z))$ separately and obtain two identical expressions. Using axiom~\ref{algebraic}, we have
$$
cc_\lambda(cc_\mu(x,y),z)=\gamma_\eta(x,y,z)
$$
where $\eta(1)=\lambda\mu, \eta(2)=\lambda(1-\mu)$ and
$\eta(3)=1-\lambda$. On the other hand, the same~\ref{algebraic} also implies
$$
cc_{\lambda\mu}(x,cc_{\frac{\lambda(1-\mu)}{1-\lambda\mu}}(y,z))=\gamma_\eta(x,y,z)
$$
with the same distribution $\eta\in \mathrm{Prob}_3$.
\end{enumerate}

We now proceed to the proof of the converse implication. Given a
family of binary operations $cc_{\lambda}$ which satisfy the axioms
of convex spaces, we first define $\gamma_{\lambda,1-\lambda}$ according to equation~(\ref{gammacc}). Given this, it then has to be shown that there exist unique choices for the $\gamma_\eta$ with $\eta\in\mathrm{Prob}_n$ for all $n\in\mathbb{N}$ such that~\ref{abelian},~\ref{linearity},~\ref{dirac} and \ref{algebraic} hold. Since $\gamma_{\iota}=\mathrm{id}_X$ for $\iota\in\mathrm{Prob}_1$, and, for $n\geq 3$, any $\eta\in\mathrm{Prob}_n$ can appear on the right-hand side of~\ref{algebraic}, we can already conclude the uniqueness: it is enough to specify the $\gamma_\eta$ with $\eta\in\mathrm{Prob}_n$ for $n=2$.

We still need to show the existence part. To this end, we first define $\gamma_\mu$ for $\mu\in \mathrm{Prob}_n$ recursively by setting
$$
\gamma_\mu(x_1,\ldots,x_n)\equiv \left\{\begin{array}{cc} x_n & \textrm{if }\mu(n)=1\\ cc_{1-\mu(n)}\left(\gamma_\nu(x_1,\ldots,x_{n-1}),x_n\right) & \textrm{if }\mu(n)\neq 1\end{array}\right.
$$
where $\nu\in\mathrm{Prob}_{n-1}$ given by
$\nu(i)=\frac{\mu(i)}{1-\mu(n)}$. So one obtains all $n$-ary operations by
repeated application of the binary ones.

Due to $cc_{\lambda}(x,y)=cc_{1-\lambda}(y,x)$, this definition
respects the permutation invariance~\ref{abelian} when the permutation does
nothing but exchange $x_1$ with $x_2$. For any $n\geq 3$, the definition can be expanded to
$$
\gamma_\mu(x_1,\ldots,x_n)=\left\{\begin{array}{cc} x_n & \textrm{if }\mu(n)=1\\ x_{n-1} & \textrm{if }\mu(n-1)=1\\ cc_{1-\mu(n)}\left(cc_{1-\frac{\mu(n-1)}{1-\mu(n)}}\left(\gamma_\eta(x_1,\ldots,x_{n-2}),x_{n-1}\right),x_n\right) & \textrm{otherwise} \end{array}\right.
$$
with $\eta\in\mathrm{Prob}_{n-2}$ given by
$\eta(i)=\frac{\nu(i)}{1-\nu(n-1)}=\frac{\mu(i)}{1-\mu(n-1)-\mu(n)}$.
Writing $y=\gamma_\eta(x_1,\ldots,x_{n-2})$, the associativity rule~\ref{associativity}
gives
\begin{align*}
\gamma_\mu(x_1,\ldots,x_n)&=cc_{1-\mu(n)}\left(cc_{1-\frac{\mu(n-1)}{1-\mu(n)}}(y,x_{n-1}),x_{n-2}\right) \\
&=cc_{1-\mu(n-1)-\mu(n)}\left(y,cc_{\frac{\mu(n-1)}{\mu(n-1)+\mu(n)}}(x_{n-1},x_n)\right)
\end{align*}
and hence~\ref{commutativity} implies the permutation invariance~\ref{abelian} also for $\gamma_\mu$ when
exchanging $x_{n-1}$ with $x_n$ while keeping all other arguments fixed. By
the recursive definition of $\gamma_\mu$, this argument also proves
invariance under transposing $x_{k-1}$ with $x_k$ for any $k<n$.
Hence now we know that~\ref{abelian} holds with respect to all transpositions
of neighboring arguments. But since the latter generate all permutations,~\ref{abelian} holds in complete generality.

With this, Brown's~\ref{linearity} and~\ref{dirac} are straightforward to prove: by~\ref{abelian},
the property~\ref{linearity} is equivalent to the analogous one with
$x_{n-1}=x_n$ instead of $x_1=x_2$. The latter follows from the
previous considerations together with the axiom $cc_\lambda(x,x)=x$.
The statement~\ref{dirac}, for $i=1$, follows directly from the definition
of $\gamma_\mu$ together with $cc_1(x,y)=y$.

Finally, we prove~\ref{algebraic} by induction on $m$. For $m=1$, this equation
coincides with our definition of its right-hand side. For $m\geq 2$, we can assume $\tilde{\mu}(m)\neq 1$ by appealing to~\ref{abelian}. Then the left-hand side of~\ref{algebraic} can be written as
$$
cc_{\nu(1)}\left(\gamma_\mu(x_1,\ldots,x_n),cc_{1-\tilde{\mu}(m)}\left(\gamma_{\mu'}(\tilde{x}_1,\ldots,\tilde{x}_{m-1}),\tilde{x}_m\right)\right)
$$
where $\mu'(i)=\frac{\tilde{\mu}(i)}{1-\tilde{\mu}(m)}$. An
application of the associativity rule~\ref{associativity} evaluates this to
$$
cc_{1-\tilde{\mu}(m)\nu(2)}\left(cc_{\frac{\nu(1)}{1-\tilde{\mu}(m)\nu(2)}}\left(\gamma_{\mu}(x_1,\ldots,x_n),\gamma_{\mu'}(\tilde{x}_1,\ldots,\tilde{x}_{m-1})\right),\tilde{x}_m\right)
$$
Now by the induction assumption, this can be written as
$$
cc_{1-\tilde{\mu}(m)\nu(2)}\left(\gamma_\delta(x_1,\ldots,x_n,\tilde{x}_1,\ldots,\tilde{x}_{m-1}),\tilde{x}_m\right)
$$
where $\delta$ is the distribution with
$\delta(i)=\frac{\nu(1)}{1-\tilde{\mu}(m)\nu(2)}\mu(i)$ for $1\leq
i\leq n$ and
$\delta(i+n)=\frac{\nu(2)}{1-\tilde{\mu}(m)\nu(2)}\tilde{\mu}(i)$.
This equation is the definition of the right-hand side of~\ref{algebraic}.
\end{proof}
\end{teo}

\section{Embeddings into vector spaces}

The following theorem and proof have been adapted from~\cite{St}.

\begin{teo}[{\cite{St}}]\label{stone}
A convex space embeds into a real vector space with~(\ref{ccdef}) if and only if the following
cancellation property holds:
$$
cc_\lambda(x,y)=cc_\lambda(x,z)\:\textrm{ with }\:\lambda\in(0,1)\quad\Longrightarrow\quad y=z\:.
$$
\end{teo}

\begin{proof}
It is clear that every convex subset of a vector space satisfies this cancellation property, so that it remains to prove the ``if'' direction.

Given a convex space $X$ with the cancellation property, we define a
real vector space as follows: let $V_X$ be the real vector space
formally generated by all points of $X$, so that $V(X)$ has a basis
$(e_x)_{x\in X}$. The vectors of the form \beq \label{quotientsub}
e_{cc_\lambda(x,y)}-\lambda e_x-(1-\lambda)e_y\:,\qquad x,y\in
X,\:\lambda\in[0,1]\:, \eeq generate a subspace $U_X\subseteq V_X$.
Let $W_X$ be the quotient space $V_X/U_X$ and let $\tilde{e}_x$
denote the image of $e_x$ under the canonical projection. Then the
mapping
$$
X\rightarrow W_X,\qquad x\mapsto\tilde{e}_x
$$
preserves convex combinations.

In order to see hat this mapping is injective, it is first necessary
to take a closer look at the subspace $U_X$. The vectors in $U_X$
are all the finite linear combinations of vectors of the
form~(\ref{quotientsub}). Taking the coefficients $\alpha_i$ and
$\beta_i$ to be non-negative, we can write such a linear combination
as
$$
\sum_{i=1}^m\alpha_i\left(e_{cc_{\lambda_i}(a_i,b_i)}-\lambda_i e_{a_i}-(1-\lambda_i)e_{b_i}\right)-\sum_{i=1}^m\beta_i\left(e_{cc_{\mu_i}(c_i,d_i)}-\mu_i e_{c_i}-(1-\mu_i)e_{d_i}\right)
$$
for certain points $a_i,b_i,c_i,d_i\in X$ and weights $\lambda_i,\mu_i\in[0,1]$. We split this into positive terms and negative terms as follows:
\beq
\label{uxvectors}
\sum_{i=1}^m\left(\alpha_i e_{cc_{\lambda_i}(a_i,b_i)}+\beta_i\mu_i e_{c_i}+\beta_i(1-\mu_i)e_{d_i}\right)-\sum_{i=1}^m\left(\beta_i e_{cc_{\mu_i}(c_i,d_i)}+\alpha_i\lambda_i e_{a_i}+\alpha_i(1-\lambda_i)e_{b_i}\right)
\eeq

This expression has two important properties: firstly, the sum of the coefficients of all negative terms equals the sum of the coefficients of all positive terms, namely $\sum_i(\alpha_i+\beta_i)$. If we assume this sum to be $1$ without loss of generality, then, secondly, both sums are just convex combinations. Interpreting these as convex combinations in $X$, these sums moreover define the same point in $X$.

We now prove the required injectivity property by showing that $\tilde{e}_x=\tilde{e}_y$ implies $x=y$ for any two points $x,y\in X$. The equation $\tilde{e}_x=\tilde{e}_y$ holds whenever $e_x-e_y$ lies in $U_X$. If this is the case, then there exists an expression of the form~(\ref{uxvectors}) where the first sum contains the term $\kappa e_x$ for some $\kappa>0$ and the second sum contains the term $\kappa e_y$ for the same $\kappa$, while all other terms cancel. Then by the above, the two sums in~(\ref{uxvectors}) define convex combinations of the same points with the same weights, except that the first one contains the point $x$ with weight $\kappa$, while the second one contains the point $y$ with weight $\kappa$. If one combines all the other points besides these $x$ and $y$ to a single point $z$ which carries a weight $1-\kappa$, one ends up with the equation
$$
cc_\kappa(x,z)=cc_\kappa(y,z)\:,
$$
which implies $x=y$ by the cancellation condition.
\end{proof}

The similarity to the Grothendieck construction which embeds a cancellative abelian monoid into an abelian group should be clear. Just like the latter proceeds by constructing a left adjoint to the inclusion functor of the category of abelian groups into the category of abelian monoids, Stone's embedding theorem implicitly constructs a left adjoint to the inclusion functor of the category of real vector spaces into the category of convex spaces.

We will soon prove that the metric compatibility axiom~\ref{metric} guarantees that the cancellation condition holds in a convex-like structure. This requires a bit of preparation:

\begin{lem}
\label{lem5}
If the equation
$$
cc_{\lambda}(y,x)=cc_{\lambda}(z,x)
$$
holds for some $x,y,z\in X$ and $\lambda\in(0,1)$, then it also holds for all $\lambda\in(0,1)$.
\end{lem}

\begin{proof}
Let us write $\lambda_0$ for the original value for which the equation holds. Then for all $\lambda<\lambda_0$,
$$
cc_{\lambda}(y,x)=cc_{\lambda/\lambda_0}(cc_{\lambda_0}(y,x),x)=cc_{\lambda/\lambda_0}(cc_{\lambda_0}(z,x),x)=cc_{\lambda}(z,x)
$$
by~\ref{associativity} and~\ref{idempotency}, so that the equation is also true in that case. Hence it is enough to find a sequence $\left(\lambda_n\right)_{n\in\mathbb{N}}$ with $\lambda_n\stackrel{n\rightarrow\infty}{\longrightarrow}1$ for which the equation holds. We construct such a sequence by defining $\lambda_{n+1}=\frac{2\lambda_n}{1+\lambda_n}$, for which an inductive argument shows the validity of the equation:
\begin{align*}
cc_{\lambda_{n+1}}(y,x)&=cc_{\lambda_n/(1+\lambda_n)}(y,cc_{\lambda_n}(y,x))=cc_{\lambda_n/(1+\lambda_n)}(y,cc_{\lambda_n}(z,x))\\[.3cm]
&=cc_{\lambda_n/(1+\lambda_n)}(z,cc_{\lambda_n}(y,x))=cc_{\lambda_n/(1+\lambda_n)}(z,cc_{\lambda_n}(z,x))=cc_{\lambda_{n+1}}(z,x)\:.
\end{align*}
\end{proof}

\begin{cor}\label{embedding}
Let $X$ be a convex-like structure. Then there is a linear embedding of $X$
into some vector space.
\begin{proof}

By Theorem \ref{equivalence} and Theorem \ref{stone} it suffices to
prove the cancellation property: if
$\gamma_{\lambda,1-\lambda}(x,y)=\gamma_{\lambda,1-\lambda}(x,z)$ for some $\lambda\in (0,1)$, then $y=z$. By the previous lemma, we know that if $\gamma_{\lambda,1-\lambda}(x,y)=\gamma_{\lambda,1-\lambda}(x,z)$ holds for some $\lambda\in(0,1)$, then it holds for all $\lambda\in(0,1)$. But then, we get from~\ref{metric}, for any $\lambda>0$,
$$
d(y,z)\leq d(y,\gamma_{\lambda,1-\lambda}(x,y))+d(z,\gamma_{\lambda,1-\lambda}(x,z))\leq \lambda d(x,y)+\lambda d(x,z)=\lambda\left[d(x,y)+d(x,z)\right]
$$
Since $\lambda$ was arbitrary, we conclude $d(y,z)=0$, and hence $y=z$.
\end{proof}
\end{cor}

\begin{rem}
The proof of Corollary \ref{embedding} indeed strongly depends on
Brown's axiom~\ref{metric}: in \cite{Fr} there are examples of
convex spaces which do not embed into a vector space.
\end{rem}

\section{Isometric embeddings into normed spaces}

\begin{lem}
\label{metricnorm}
Let $(X,d)$ be a metric space which is a convex subset $X\subseteq E$ of some vector space $E$ such that
\begin{equation}
\label{metricconv}
d(\lambda y+(1-\lambda)x,\lambda z+(1-\lambda)x)\leq \lambda d(y,z)\quad\forall x,y\in X,\:\lambda\in[0,1]
\end{equation}
holds. Then there is a norm $||\cdot||$ on $E$ such that for
all $x,y\in X$,
$$
d(x,y)=||x-y||\:.
$$
\end{lem}

\begin{proof}
As a special case,~(\ref{metricconv}) gives for $z=x$,
$$
d(\lambda y+(1-\lambda)x,x)\leq \lambda d(y,x)
$$
which yields, in combination with the triangle inequality,
$$
d(y,x)\leq d(y,\lambda y+(1-\lambda)x)+d(\lambda y+(1-\lambda)x,x)\leq (1-\lambda) d(y,x)+\lambda d(y,x)\:.
$$
Since the term on the left-hand side equals the term on the right-hand side, we deduce that both inequalities are actually equalities. In particular, the metric is ``uniform on lines'' in the sense that
$$
d(x,(1-\lambda)x+\lambda y)=\lambda d(x,y)\quad\forall x,y\in X,\:\lambda\in [0,1]\:.
$$

Now in order to prove the assertion, it needs to be shown that $d$
is translation-invariant in the following sense: suppose that
$x_0,x_1,y_0,y_1\in X$
are such that
$$
y_1-x_1=y_0-x_0\:,
$$
then $d(x_1,y_1)=d(x_0,y_0)$. See figure~\ref{parallelogram} for an illustration. For $\varepsilon\in(0,1)$, we will
also consider the points
$$
x_{\varepsilon}=\varepsilon x_1+(1-\varepsilon)x_0\:,\qquad y_{\varepsilon}=\varepsilon y_1+(1-\varepsilon)y_0\:,\qquad z_\varepsilon=(1-\varepsilon)x_\varepsilon + \varepsilon y_\varepsilon=\varepsilon y_1+(1-\varepsilon)x_0\:.
$$
Then by the assumption~(\ref{metricconv}),
$$
d(x_{\varepsilon},z_\varepsilon)=d\left(\varepsilon x_1+(1-\varepsilon)x_0,\varepsilon y_1+(1-\varepsilon)x_0\right)\leq \varepsilon d(x_1,y_1)\:.
$$
By the definition of $z_\varepsilon$ and the uniformity of $d$ on the line connecting $z_\varepsilon$ with $x_\varepsilon$ and $y_\varepsilon$, we have
$$
d(x_\varepsilon,y_\varepsilon)=\varepsilon^{-1}d(x_\varepsilon,z_\varepsilon)\leq d(x_1,y_1)\:.
$$
Upon taking the limit $\varepsilon\rightarrow 0$ we therefore arrive at
$$
d(x_0,y_0)\leq d(x_1,y_1)\:,
$$
and the other inequality direction is then clear by symmetry, so that $d$ is indeed translation invariant.

Now $d$ can be uniquely extended to a translation-invariant metric on the affine hull of $X$. Assuming $0\in X$ without loss of generality, this affine hull equals the linear hull, and then the translation-invariant metric on $\mathrm{lin}(X)$ comes from a norm. If necessary, this norm can be extended from the subspace $\mathrm{lin}(X)$ to all of $E$.
\end{proof}

\begin{figure}
\begin{centering}
\begin{tikzpicture}
\draw (0,0)node[anchor=north east]{$x_0$}--(5,1)node[anchor=north west]{$x_1$}--(5,4)node[anchor=south west]{$y_1$}--(0,3)node[anchor=south east]{$y_0$}--(0,0)--(1,.2)node[anchor=north]{$x_\varepsilon$}--(1,3.2)node[anchor=south]{$y_\varepsilon$};
\draw (0,0)--(5,4);
\draw (1,.8)node[anchor=south east]{$z_\varepsilon$};
\fill (0,0) circle (.05);
\fill (5,1) circle (.05);
\fill (5,4) circle (.05);
\fill (0,3) circle (.05);
\fill (1,.2) circle (.05);
\fill (1,3.2) circle (.05);
\fill (1,.8) circle (.05);
\end{tikzpicture}
\end{centering}
\label{parallelogram}
\caption{Illustration of the proof of lemma~\ref{metricnorm}.}
\end{figure}
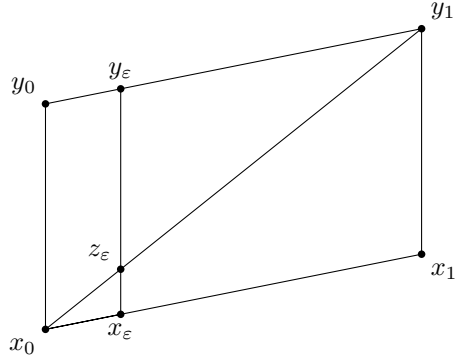

Now we have assembled all the ingredients for our main theorem:

\begin{teo}\label{banach}
Every convex-like structure is affinely and isometrically isomorphic
to a closed convex subset of a Banach space.
\end{teo}

\begin{proof}
Since the inequality~(\ref{metricconv}) is an instance of the metric compatibility axiom~\ref{metric}, this is a direct consequence of corollary~\ref{embedding} and lemma~\ref{metricnorm} and the fact that every norm space embeds into its completion, which is a Banach space. Closedness then follows from the requirement that a convex-like structure is assumed to be complete.
\end{proof}

\begin{rem}
We have not used the completeness of $X$ in the derivation of
corollary~\ref{embedding} or lemma~(\ref{metricnorm}). So if we would remove this
hypothesis from the axioms, then we would get that (not necessarily complete) convex-like structures are
precisely the convex subsets of normed spaces.
\end{rem}

\begin{rem}
\label{finalrem}
Given the axioms in Definition \ref{convex-like}, Brown's original first metric compatibility condition
\begin{quote}
``\: There is a constant $C$ such that for all $x_1,\ldots, x_n \in X$,
$$d (\gamma_{\mu} (x_1,\ldots,x_n), \gamma_{\tilde{\mu}}
(x_1,\ldots,x_n)) \leq C \sum_{i=1}^n| \mu(i) - \tilde{\mu}(i)|\:,\textrm{''}
$$
\end{quote}
holds if and only if $X$ is bounded (as a metric space).
\end{rem}

\begin{proof}
By Theorem \ref{banach}, we can take $X$ to be a closed convex subset of a Banach space, with the metric $d$ induced by the norm. Brown's condition then just states that
\begin{equation}
\label{firstmetricaxiom}
\left\|\sum_{i=1}^n(\mu(i)-\tilde{\mu}(i))\, x_i\right\|\leq C\sum_{i=1}^n| \mu(i) - \tilde{\mu}(i)|\:.
\end{equation}
If $X$ is bounded, then we can set $C=\sup_{x\in X}||x||$, and the inequality holds. Conversely, we can use~(\ref{firstmetricaxiom}) to deduce the boundedness of $X$: taking $n=2$ and $\mu(1)=\tilde{\mu}(2)=1$ gives
$$
d(x_1,x_2)=||x_1-x_2||\leq 2C\:.
$$
\end{proof}

The following corollary is a reformulation of our previous results.
It provides a simple way to axiomatize (closed) convex subsets of Banach
spaces.

\begin{cor}
Let $(X,\{cc_\lambda\},d)$ be a convex space in the sense of Definition
\ref{convex} together with a (complete) metric $d$. It is a (closed)
convex subset of a Banach space if and only if it satisfies the inequality
$$
d(cc_\lambda(y,x),cc_\lambda(z,x)) \leq \lambda d(y,z) \quad\forall x,y,z\in X\:,\lambda\in[0,1]\:.
$$
\end{cor}

\end{document}